\documentclass[a4paper,reqno]{amsart}
\usepackage{amssymb, amsmath, amscd}
\usepackage{color}

\newtheorem{theorem}{Theorem}[section]

\newtheorem{lemma}{Lemma}[section]

\makeatletter
 \@addtoreset{equation}{section}
\makeatother
\begin{document}
\title{Universal curvature identities}
\author{P. Gilkey, J.H. Park and K. Sekigawa}
\address{PG: Department of Mathematics, University of Oregon, Eugene OR 97403 USA\\ E-mail: gilkey@uoregon.edu}
\address{JHP: Department of Mathematics, Sungkyunkwan University, Suwon 440-746, Korea.\\
 E-mail: parkj@skku.edu}
\address{KS: Department of Mathematics, Faculty of Science, Niigata University, Niigata, 950-2181, Japan.
E-mail: sekigawa@math.sc.niigata-u.ac.jp}
\begin{abstract}{We study scalar and symmetric 2-form valued universal curvature identities. We use this to establish the
Gauss-Bonnet theorem using heat equation methods, to give a new proof of a result of Kuz'mina and Labbi concerning the
Euler-Lagrange equations of the Gauss-Bonnet integral, and to give a new derivation of the Euh-Park-Sekigawa identity.
}\end{abstract}

\maketitle

\section{Introduction and outline of paper}
The study of Riemannian geometry relies to a large extent on the examination of curvature and of local curvature invariants of
the manifold both for their own sake but also in relationship to other structures (see, for example,
\cite{B10, B9, BGKNW9, EPS9,
K10,Ko10, N10, O9}) - this paper follows in that line of investigation.

\subsection{Scalar invariants of the metric} Let $\mathcal{I}_{m,n}$ be the space of
scalar invariant local formulas which are homogeneous of
order $n$ in the derivatives of the metric and which are defined in
the category of all Riemannian manifolds of dimension $m$; we refer
to Section \ref{sect-2} for details. Since $\mathcal{I}_{m,n}=\{0\}$
if $n$ is odd, we shall assume $n$ even henceforth. Such invariants
are given by contracting indices in monomials involving the
covariant derivatives of the curvature tensor. Let $R_{ijkl}$ be the
components of the curvature tensor relative to a local orthonormal
frame $\{e_1,...,e_m\}$ for the tangent bundle of $M$. For example,
the scalar curvature may be defined by setting:
$$\tau_m:=\sum_{i,j=1}^mR_{ijji}\in\mathcal{I}_{m,2}\,.$$
There is a natural restriction map $r:\mathcal{I}_{m,n}\rightarrow\mathcal{I}_{m-1,n}$ given by restricting the summation to
range from $1$ to $m-1$ that will be discussed in Section \ref{sect-2}. For example, we have that $r(\tau_m)=\tau_{m-1}$. Thus
the scalar curvature is {\it universal} and for that reason it is not usually subscripted in this fashion. More generally, we
have (see, for example, the discussion in \cite{G94}) the following universal spanning sets for $n=0,2,4,6$; we shall suppress
the role of the dimension $m$ to simplify the notation and we shall adopt the {\it Einstein convention} and sum over repeated
indices.  Let $\rho$ be the Ricci tensor and let $R$ be the full curvature
tensor.
\goodbreak\begin{lemma}\label{lem-1.1}
\ \begin{enumerate}
\item
$\mathcal{I}_{m,0}=\operatorname{Span}\left\{1\right\}$.
\smallbreak\item
$\mathcal{I}_{m,2}=\operatorname{Span}\left\{\tau:=R_{ijji}\right\}$.
\smallbreak\item
$\mathcal{I}_{m,4}=\operatorname{Span}\left\{\Delta\tau:=-R_{ijji;kk},\tau^2:=R_{ijji}R_{kllk},
       |\rho|^2:=R_{ijjk}R_{illk}\right.$,
\smallbreak $\qquad\qquad\left.|R|^2:=R_{ijkl}R_{ijkl}\right\}$.
\smallbreak\item
$\mathcal{I}_{m,6}=\operatorname{Span}\left\{R_{ijji;kkll},R_{ijji;k}R_{lnnl;k},R_{aija;k}R_{bijb;k}\right.$,
$R_{ajka;n}R_{bjnb;k},$\smallbreak
$R_{ijkl;n}R_{ijkl;n},R_{ijji}R_{kllk;nn},R_{ajka}R_{bjkb;nn},R_{ajka}R_{bjnb;kn}$,
\smallbreak
$R_{ijkl}R_{ijkl;nn},R_{ijji}R_{kllk}R_{abba},R_{ijji}R_{ajka}R_{bjkb},
R_{ijji}R_{abcd}R_{abcd}$,
\smallbreak
$R_{ajka}R_{bjnb}R_{cknc},R_{aija}R_{bklb}R_{ikjl},
R_{ajka}R_{jnli}R_{knli},R_{ijkn}R_{ijlp}R_{knlp}$,
\smallbreak
$\left.R_{ijkn}R_{ilkp}R_{jlnp}\right\}$.
\end{enumerate}\end{lemma}

Lemma \ref{lem-1.1} follows from Lemma \ref{lem-2.2} (see Section
\ref{sect-2}) with a bit of work; we shall omit details as we shall
not need Lemma \ref{lem-1.1} in what follows and simply present it
for the purposes of illustration. The universal scalar invariants
given in Lemma \ref{lem-1.1} are linearly independent if $m\ge n$.
However, they are not linearly independent if $m=n-1$ and there is a
single additional universal relation amongst these invariants that
we may describe as follows. Define the {\it Pfaffian}
$E_{m,n}\in\mathcal{I}_{m,n}$ for $n$ even by setting:
$$E_{m,n}:=\sum_{i_1,...,i_n,j_1,...,j_n=1}^m{{R_{i_1i_2j_2j_1}...R_{i_{n-1}i_{n}j_nj_{n-1}}}}g(e^{i_1}\wedge...\wedge
e^{i_n},e^{j_1}\wedge...\wedge e^{j_n})\,.$$
For example,
$E_{m,2}=2\tau_m$ is essentially just the scalar curvature. The
invariants $E_{m,n}$ are again {\it universal}, i.e.
$$E_{m,n}\in\mathcal{I}_{m,n}\quad\text{and}\quad r(E_{m,n})=E_{m-1,n}\,.$$
It is also immediate that $r(E_{m,m})=0$ since
$e^{i_1}\wedge...\wedge e^{i_m}$ vanishes on a manifold of dimension
$m-1$. Consequently,
$E_{m,m}\in\ker(r:\mathcal{I}_{m,m}\rightarrow\mathcal{I}_{{m-1,m}})$
and $E_{m,m}$ provides a universal relation in curvature. Expressing
the invariants $E_{m,2}$, $E_{m,4}$, and $E_{m,6}$ universally in
terms of contractions of indices (see, for example, the discussion
in \cite{P88}) then yields the following relations:
\goodbreak\begin{lemma}\label{lem-1.2}
\ \begin{enumerate}
\item If $m=1$, then $0=R_{ijji}$.
\smallbreak\item If $m=3$, then $0=R_{ijji}R_{kllk}-4R_{aija}R_{bijb}+R_{ijkl}R_{ijkl}$.
\smallbreak\item If $m=5$, then $0=R_{ijji}R_{kllk}R_{abba}-12R_{ijji}R_{aija}R_{bijb}+3R_{abba}R_{ijkl}R_{ijkl}$
\smallbreak\qquad\quad
$+24R_{aija}R_{bklb}R_{jlik}+16R_{aija}R_{bjkb}R_{cikc}-24R_{aija}R_{jkln}R_{lnik}$
\smallbreak\qquad\quad
$+2R_{ijkl}R_{klan}R_{anij}-8R_{kaij}R_{inkl}R_{jlan}$.
\end{enumerate}\end{lemma}
 In fact, these the only such universal relations of this type \cite{G73}:

\goodbreak\begin{theorem}\label{thm-1.1}
\
\begin{enumerate}
\item  $r:\mathcal{I}_{m,n}\rightarrow\mathcal{I}_{m-1,n}$ is always
surjective.
\smallbreak\item If $n$ is even and if $m>n$, then $r:\mathcal{I}_{m,n}\rightarrow\mathcal{I}_{m-1,n}$ is bijective.
\smallbreak\item Let $m$ be even. Then $\ker\{r:\mathcal{I}_{m,m}\rightarrow\mathcal{I}_{m-1,m}\}=E_{m,m}\cdot\mathbb{R}$.
\end{enumerate}
\end{theorem}

\subsection{Heat trace asymptotics}
Theorem \ref{thm-1.1} was originally established to provide a heat equation proof of the Gauss-Bonnet Theorem \cite{G73}.
We sketch the derivation to illustrate the use of Theorem \ref{thm-1.1}. Let $(M,g)$
be a compact Riemannian manifold. Let
$\Delta_p$ be the Laplacian on
$p$-forms. The fundamental solution of the heat equation $e^{-t\Delta_p}$ is of trace class. If $f\in C^\infty(M)$, then there
is a complete asymptotic series as
$t\downarrow0$ of the form
$$\operatorname{Tr}_{L^2}(fe^{-t\Delta_p})\sim\sum_{n=0}^\infty t^{(n-m)/2}\int_Mf(x)a_{m,n,p}(x,\Delta_p)d\nu$$
where $a_{m,n,p}\in\mathcal{I}_{m,n}$ is a local invariant which is
homogeneous of order $n$ in the jets of the metric and where $d\nu$ is the {Riemannian measure}:
$$d\nu=gdx^1...dx^m\quad\text{where}\quad
g=\sqrt{\det(g_{ij})}\quad\text{and}\quad
g_{ij}=g(\partial_{x^i},\partial_{x^j})\,.$$Note that $a_{m,n,p}=0$
if $n$ is odd. We take the super trace and set
$$a_{m,n}:=\sum_{p=0}^m(-1)^pa_{m,n,p}\in\mathcal{I}_{m,n}\,.$$
The cancellation argument of Bott \cite{AB67}
shows that we have a local formula for the Euler-Poincar\'e characteristic:
$$\chi(M)=\int_Ma_{m,m}(x)d\nu\,.$$
It also follows using suitable product formulas that $r(a_{m,n})=0$ for any $(m,n)$. Let $m$ be even ($\chi(M)=0$ if $m$ is
odd). Theorem \ref{thm-1.1} implies that there is a universal constant $c_m$ so that
\medbreak\noindent\hfill
$a_{m,m}=\left\{\begin{array}{lll}0&\text{if}&n<m\\
c_mE_{m,m}&\text{if}&n=m\end{array}\right\}\quad\text{and thus}\quad
\displaystyle\chi(M)=\int_Mc_mE_{m,m}$.\hfill\vphantom{.}
\medbreak\noindent
The constant is easily determined by evaluation on the manifold $S^2\times...\times S^2$ and the Gauss-Bonnet formula results.
We remark in passing that it is possible to examine $\ker(r:\mathcal{I}_{m,m+2}\rightarrow\mathcal{I}_{m-1,m+2})$ and
thereby evaluate the next term in the heat expansion $a_{m,m+2}$ \cite{G79}.

\subsection{Symmetric 2-tensor valued invariants} Let $\mathcal{I}_{m,n}^2$ be the space of symmetric $2$-form valued
invariants which are homogeneous of degre $n$ in the derivatives of the metric and which are defined in the category of $m$
dimensional Riemannian manifolds; again we refer to Section \ref{sect-2} for further details. Let $\{e_1,...,e_k\}$ be a local
orthonormal frame for the tangent bundle of $M$. If
$\xi$ and
$\eta$ are cotangent vectors, then the symmetric product is denoted by
$\xi\circ\eta:=\frac12\{\xi\otimes\eta+\eta\otimes\xi\}$. For example, $g=e^k\circ e^k$.
One has:

\goodbreak\begin{lemma}\label{lem-1.3}
\ \begin{enumerate}
\item $\mathcal{I}_{m,0}^2=\operatorname{Span}\left\{e^k\circ e^k\right\}$.
\smallbreak\item
$\mathcal{I}_{m,2}^2=\operatorname{Span}\left\{R_{ijji}e^k\circ
e^k,R_{ijki}e^j\circ e^k\right\}$. \smallbreak\item
$\mathcal{I}_{m,4}^2=\operatorname{Span}\left\{R_{ijji;kk}e^l\circ
e^l\right.$, $R_{kjjl;ii}e^k\circ e^l$, $R_{ijji;kl}e^k\circ e^l$,
\smallbreak\quad
$R_{ijji}R_{kllk}e^n\circ e^n$,
$R_{ijki}R_{ljkl}e^n\circ e^n$,
$R_{ijkl}R_{ijkl}e^n\circ e^n$, $R_{ijji}R_{klnk}e^l\circ e^n$,
\smallbreak\quad
$R_{ikli}R_{jknj}e^l\circ e^n$,
$R_{ijkl}R_{ijkn}e^l\circ e^n$, $\left.R_{lijn}R_{kijk}e^l\circ e^n\right\}$.
\end{enumerate}
\end{lemma}

 Lemma \ref{lem-1.3} also follows from Lemma \ref{lem-2.2}
and again we shall omit details as we shall not need this result in
what follows and simply present it for the purposes of illustration.

Restricting the range of summation and setting $e^j\circ e^k=0$ if $j=m$ or if $k=m$ yields an analogous restriction map
$r:\mathcal{I}_{m,n}^2\rightarrow\mathcal{I}_{m-1,n}^2$; the elements given in Lemma \ref{lem-1.3} are universal with respect
to restriction. They are linearly independent if $m>n$, but there is a single relation if $m=n$ we may describe as follows.
For $n$ even, define
$ T_{m,n}^2\in\mathcal{I}_{m,n}^2$ by setting:
$$
\begin{array}{l}
\displaystyle
T_{m,n}^2:=\sum_{i_1,...,i_{n+1},j_1,...,j_{n+1}=1}^mR_{i_1i_2j_2j_1}...R_{i_{n-1}i_{n}j_{n}j_{n-1}}e^{i_{n+1}}\circ
e^{j_{n+1}}\\
\qquad\qquad\qquad\qquad\times
g(e^{i_1}\wedge...\wedge e^{i_{n+1}},e^{j_1}\wedge...\wedge e^{j_{n+1}})\,.\vphantom{\vrule height 11pt}
\end{array}$$

 It is then immediate that $r( T_{m,n}^2)=T_{m-1,n}^2$ so these
elements are again universal. Furthermore, we again have that
$r( T_{m+1,m}^2)=0$. This then leads to the identities:
\begin{lemma}\label{lem-1.4}
\ \begin{enumerate}
\item If $m=2$, then $0=R_{ijji}e^k\circ e^k-2R_{ijki}e^j\circ e^k$.
\smallbreak\item If $m=4$, then $0=-\frac14(R_{ijji}R_{kllk}-4R_{ijki}R_{ljkl}+R_{ijkl}R_{ijkl})e^n\circ e^n$
\smallbreak $+\{R_{klni}R_{klnj}-2R_{knik}R_{lnjl}
-2R_{iklj}R_{nkln}+R_{kllk}R_{nijn}\}e^i\circ e^j$.
\end{enumerate}\end{lemma}

In fact the identities of Lemma \ref{lem-1.4} are the only universal identities of this form if $m=2$ or if $m=4$. In Section
\ref{sect-2}, we will establish the following extension of Theorem \ref{thm-1.1}; this is the main new result of this paper:

\goodbreak\begin{theorem}\label{thm-1.2}
\
\begin{enumerate}
\item $r:\mathcal{I}_{m,n}^2\rightarrow\mathcal{I}_{m-1,n}^2$ is always
surjective.
\smallbreak\item If $n$ is even and if $m>n+1$, then $r:\mathcal{I}_{m,n}^2\rightarrow\mathcal{I}_{m-1,n}^2$ is
bijective.
\smallbreak\item If $m$ is even, then
$\ker\{r:\mathcal{I}_{m+1,m}^2\rightarrow\mathcal{I}_{m,m}^2\}=
T_{m+1,m}^2\cdot\mathbb{R}$.
\end{enumerate}
\end{theorem}

It is worth presenting an example to illustrate the use of Theorem 1.2. Let
$m=2$.  Then
$T_{3,2}^2\in
\mathcal{I}_{3,2}^2$ is defined by setting:
$$
T_{3,2}^2=\sum_{i_1,i_2,i_{3},j_1,j_2,j_{3}=1}^3R_{i_1i_2j_2j_1}e^{i_{3}}\circ
e^{j_{3}}\times g(e^{i_1}\wedge e^{i_2}\wedge
e^{i_{3}},e^{j_1}\wedge e^{j_2}\wedge e^{j_{3}})\,.$$
Then Theorem \ref{thm-1.2} (3) yields the relation:
$$
0 = r(T_{3,2}^2) = 2\sum_{i,j, k=1}^2R_{ijji} e^k \circ e^k
- 4 \sum_{i,j, k=1}^2R_{kijk} e^i \circ e^j\,.
$$
This implies the following well-known curvature identity on any 2-dimension
Riemannian manifold $$\rho =\frac{1}{2}\tau_2 g\,.$$

\subsection{Euler-Lagrange Equations}
As was the case for Theorem \ref{thm-1.1}, Theorem \ref{thm-1.2} is
motivated by index theory. Let $h$ be an arbitrary symmetric
$2$-tensor field. We form the $1$-parameter family of metrics
$g(\varepsilon):=g+\varepsilon h$. Since $E_{m,n}$ only involves the
first and second derivatives of the metric, the variation only
involves the first and second derivatives of $h$. We may therefore
express \medbreak\noindent\hfill
$\displaystyle\partial_\varepsilon\left.\left\{E_{m,n}(g(\varepsilon))d\nu_{g(\varepsilon)}\right\}\right|_{\varepsilon=0}
=Q^{m,n}_{ij}h_{ij}+Q^{m,n}_{ijk}h_{ij;k}+Q^{m,n}_{ijkl}h_{ij;kl}$.\hfill\vphantom{.}
\medbreak\noindent where $h_{ij;k}$ and $h_{ij;kl}$ give the
components of the covariant derivative of $h$ with respect to the
Levi-Civita connection of $g$ and where we write write $(m,n)$ as a
super script on $Q$ to avoid notational complexity.  Let
$Q_{ijk;l}^{m,n}$ and $Q_{ijkl;uv}^{m,n}$ be the components of the
first and second covariant derivatives of these tensors,
respectively.  Define:
$${S^2_{m,n}:=\{Q^{m,n}_{ij}-Q^{m,n}_{ijk;k}+Q^{m,n}_{ijkl;lk}\}e^i\circ e^j}\,.$$
It is then immediate from the definition that
$$S_{m,n}^2\in\mathcal{I}_{m,n}^2\quad\text{and}\quad r(S_{m,n}^2)=S_{m-1,n}^2\,.$$
This tensor is characterized by the property that if $(M,g)$ is any
compact Riemannian manifold of dimension $m$, then we may integrate
by parts to see that:
$$\partial_\varepsilon\left.\left\{\int_ME_{m,n}(g(\varepsilon))d\nu_{g(\varepsilon)}\right\}\right|_{\varepsilon=0}
=\int_MS_{m,n,ij}^2h_{ij}d\nu(g)\,.
$$
The Gauss-Bonnet theorem shows that this vanishes if $m=n$. Therefore
$$
  S_{m+1,m}^2\in\ker(r:\mathcal{I}_{m+1,m}^2\rightarrow\mathcal{I}_{m,m}^2)
  \quad\text{and thus}\quad S_{m+1,m}^2=d_m T_{m+1,m}^2\,.
$$
In
particular, we establish a conjecture of Berger \cite{M70} that
$S_{m,n}^2$ involves only the second derivatives of the
metric. This result is, of course, not new. It was first established
by Kuz'mina \cite{K74} and subsequently established using different
methods by Labbi \cite{L05,L07,L08}. It is at the heart of recent
work in 4-dimensional geometry \cite{EPS10,EPS10a,EPS10b,EPS10c}.

\subsection{Outline of the paper} In Section \ref{sect-2}, we shall define the spaces $\mathcal{I}_{m,n}$ and
$\mathcal{I}_{m,n}^2$. We shall discuss the restriction map and derive its elementary properties. We review the first theorem
 of H. Weyl \cite{W46} on the invariants of the orthogonal group.
These are used in Lemma \ref{lem-2.3} to show that $r$
is surjective; this establishes Assertion (1) of Theorem \ref{thm-1.1} and of
Theorem \ref{thm-1.2}. We will continue our study and complete the proof
of Assertion (2) of Theorem \ref{thm-1.1} and of Theorem \ref{thm-1.2} in Lemma
\ref{lem-2.5}. We then use the second theorem of H. Weyl on the
invariants of the orthogonal group to establish Assertion (2) of
Theorem \ref{thm-1.1} and of Theorem \ref{thm-1.2}.

We remark the the generalization of Theorem \ref{thm-1.1}
\cite{G73a} to the complex setting yields a heat equation proof of
the Riemann-Roch theorem for K\"ahler manifolds; it would be
interesting to know if there is a suitable generalization of Theorem
\ref{thm-1.2} to the K\"ahler setting that could be used to study
the associated Euler-Lagrange equations for the Chern numbers.

\section{Invariance theory}\label{sect-2}

In this section, we review the basic results of invariance theory
that we shall need. We work non-classically in Section
\ref{sect-2.1} and use the derivatives of the metric rather than the
Riemann curvature tensor to define the space $\mathcal{I}_{m,n}$ of
scalar invariant local formulas  and the space $\mathcal{I}_{m,n}^2$
of symmetric 2-tensor valued invariant local formulas which are
homogeneous of degree $n$ in the jets of the metric in the category
of $m$-dimensional Riemannian manifolds. In Section \ref{sect-2.2},
we give a more classical treatment using the Riemann curvature
tensor. In Section \ref{sect-2.3} we review the first theorem of
invariants of  H. Weyl \cite{W46}. In Section \ref{sect-2.4}, we
discuss the restriction map and establish in Lemma \ref{lem-2.3}
that $r$ is surjective. In Lemma \ref{lem-2.5} we show that
$\ker(r:\mathcal{I}_{m,n}\rightarrow\mathcal{I}_{m-1,n})=\{0\}$ if
$m>n$ (resp. that
$\ker(r:\mathcal{I}_{m,n}^2\rightarrow\mathcal{I}_{m-1,n}^2)=\{0\}$
if $m>n+1$). We also derive some results in the limiting case $m=n$
(resp. $m=n+1$) that will be useful subsequently. In Section
\ref{sect-2.5} we recall H. Weyl's second theorem of invariants;
this result is used in Section \ref{sect-2.6} to complete the proof
of Theorem \ref{thm-1.1} and in Section \ref{sect-2.7} to complete
the proof of Theorem \ref{thm-1.2}. This approach is a bit different
from that used in \cite{G73} and is, we believe, more instructive.

\subsection{Local scalar invariants of the metric}\label{sect-2.1} We follow the discussion in \cite{G73}
to
establish Theorem \ref{thm-1.1}.
 Let $\delta_i^j$ and $\delta_{ij}$ be the Kronecker symbols;
$$\delta_i^j=\delta_{ij}=\left\{\begin{array}{lll}0&\text{if}&i\ne j,\\1&\text{if}&i=j\end{array}\right\}\,.$$
 Fix a dimension
$m$. Let
$\alpha=(a_1,...,a_m)$ be a non-trivial multi-index where the $a_i=\alpha(i)$
are non-negative integers not all of which vanish. Introduce formal variables
$$\{g_{ij}=g_{ji},g^{ij}=g^{ji},g,g_{ij/\alpha}=g_{ji/\alpha}\}\quad\text{for}\quad 1\le i,j\le m\,.$$
Let $\mathcal{Q}_m$ be the free commutative unital $\mathbb{R}$ algebra generated by these
variables where we impose the obvious relationships:
$$\sum_{k=1}^mg_{ik}g^{jk}=\delta_i^j\quad\text{and}\quad
\det(g_{ij})=g^2\,;$$ $\mathcal{Q}_m$ is the algebra of {\it local
formulae in the derivatives of the metric}. Given a system of local
coordinates $\vec x=(x^1,...,x^m)$ defined near a point $P$ of a Riemannian
manifold
$(M,g)$, let $\partial_{x^i}:=\frac{\partial}{\partial x^i}$. It will also be
convenient to introduce the following notation for the first and
second derivatives of the metric: \medbreak\hfill
$g_{ij/k}:=\partial_{x^k}g_{ij}\quad\text{and}\quad
 g_{ij/kl}:=\partial_{x^k}\partial_{x^l}g_{ij}$.\hfill\phantom{.}
 \medbreak If
$Q\in\mathcal{Q}_m$, then we shall define $Q(\vec
x,g,P)\in\mathbb{R}$ by substitution setting:
$$\begin{array}{rr}
g_{ij}(\vec x,g,P):=g(\partial_{x^i},\partial_{x^j})(P),&
  g^{ij}(\vec x,g,P):=g(dx^i,dx^j)(P),\\
g(\vec x,g,P):=\det\{g_{ij}(\vec x,g,P)\}^{1/2},&
  g_{ij/\alpha}(\vec x,g,P):=\partial_{x^1}^{a_1}\cdot\cdot\cdot\partial_{x^m}^{a_m}g_{ij}(\vec x,g,P).
\vphantom{\vrule height 11pt}\end{array}$$
We say that $Q$ is {\it invariant} if $Q(\vec x,g,P)$ is independent of the coordinate system $\vec x$ for every
possible such $(M,g,P)$; we denote this common value by $Q(g,P)$ and let $\mathcal{I}_m$ be the vector space of all such
invariant local formulae.

We define the {\it weight} of $g_{ij/\alpha}$ to be $|\alpha|:=a_1+...+a_m$ and the weight of $\{g_{ij},g^{ij},g\}$ to be zero.
Let $\mathcal{I}_{m,n}\subset\mathcal{I}_m$ be the space of invariant local formulas which are
weighted homogeneous of order $n$. One can use dimensional analysis to establish \cite{G73} that:
\begin{lemma}\label{lem-2.1}
Let $Q\in\mathcal{I}_m$. Then $Q\in\mathcal{I}_{m,n}$ if and only if $Q(c^2g,P)=c^{-n}Q(g,P)$ for all $0\ne c\in\mathbb{R}$
and all $(M,g,P)$.
\end{lemma}

As a consequence of Lemma \ref{lem-2.1}, we may decompose $\mathcal{I}_m=\oplus_n\mathcal{I}_{m,n}$ as
the graded direct sum of the formulae which are weighted homogeneous of degree
$n$. Furthermore, by taking $c=-1$, we see that $\mathcal{I}_{m,n}=\{0\}$ if
$n$ is odd and we shall restrict to the case $n$ even henceforth.

Next, we consider a local formula
$$Q=\sum_{i,j=1}^mQ_{ij}dx^i\circ dx^j$$
where the $Q_{ij}\in{\mathcal{Q}_m}$. Evaluation is defined as
above and we say $Q$ is invariant if $Q(\vec x,g,P)$ is independent
of $\vec x$ for all $(g,P)$. We let $\mathcal{I}_m^2$ be the space
of all such invariant local formulas. The obvious generalization of
Lemma \ref{lem-2.1} permits us to decompose
$\mathcal{I}_m^2=\oplus_n\mathcal{I}_{m,n}^2$ where
$\mathcal{I}_{m,n}^2$ consists of those invariant local formulas
which are homogeneous of degree $n$ in the jets of the metric.
Again, $\mathcal{I}_{m,n}^2=\{0\}$ if $n$ is odd.

\subsection{The Riemann curvature tensor}\label{sect-2.2}
Although convenient for our subsequent purposes, the definition of local invariants given in Section \ref{sect-2.1} is
non-classical and it is worth making contact with the more standard approach. Let
$\nabla$ be the Levi-Civita connection of a Riemannian manifold
$(M,g)$.  The associated {\it Christoffel symbols} are defined in a system of local coordinates by setting:
$$
\nabla_{\partial_{x^i}}\partial_{x^j}=\Gamma_{ij}{}^k\partial_{x^k}\text{
where }
\Gamma_{ij}{}^k:=\textstyle\frac12g^{kl}(\partial_{x^i}g_{jl}
+\partial_{x^j}g_{il}-\partial_{x^l}g_{ij})\,.
$$
The Riemann curvature tensor $R_{ijk}{}^l$, the Ricci tensor $\rho$, the scalar curvature $\tau$, the norm
$|\rho|^2$ of the Ricci tensor, and the norm $|R|^2$ of $R$ are then given by:
\begin{eqnarray}
&&R_{ijk}{}^l:=\partial_{x^i}\Gamma_{jk}{}^l-\partial_{x^j}\Gamma_{ik}{}^l+
\Gamma_{in}{}^l\Gamma_{jk}{}^n-\Gamma_{jn}{}^l\Gamma_{ik}{}^n,\nonumber\\
&&\rho_{jk}:=R_{ijk}{}^i,\qquad\tau:=g^{i_1j_1}\rho_{i_1j_1},\quad
|\rho|^2:=g^{i_1j_1}g^{i_2j_2}\rho_{i_1i_2}\rho_{j_1j_2},\label{eqn-2.a}\\
&&|R|^2:=g^{i_1j_1}g^{i_2j_2}g^{i_3j_3}g_{i_4j_4}
    R_{i_1i_2i_3}{}^{i_4}R_{j_1j_2j_3}{}^{j_4}\,.\nonumber
\end{eqnarray}
Again, we really should subscript to indicate the dependence on the
dimension $m$ explicitly in the Einstein summations but we will omit
this additional notational complexity in the interests of brevity as
the formulas are universal and no confusion will result from this
notational imprecision. Since $\Gamma$ has weight 1 and
$\partial_{x^i}\Gamma$ has weight 2, we see that $R$ has weight 2.
Consequently,
$$\tau\in\mathcal{I}_{m,2},\quad\tau^2\in\mathcal{I}_{m,4},\quad
|\rho^2|\in\mathcal{I}_{m,4},\quad|R|^2\in\mathcal{I}_{m,4}\,.
$$
We let ``;" denote multiple covariant differentiation. If $\Delta$ is the scalar Laplacian, we have
$$\Delta\tau=-g^{ij}\tau_{;ij}\in\mathcal{I}_{m,4}\,.$$

\subsection{H. Weyl's Theorem of invariants}\label{sect-2.3}
Let $V$ be a finite dimensional vector space which is equipped with a positive
definite bilinear form
$\langle\cdot,\cdot\rangle$ of signature $(p,q)$. Let $\mathcal{O}$ be the associated orthogonal group. We say that
$\psi:\otimes^kV^*\rightarrow\mathbb{R}$ is a {\it linear orthogonal invariant}\index{linear orthogonal
invariant} if
$\psi$ is a linear map and if
$$\psi(\Theta\cdot w)=\psi(w)\quad\forall \Theta\in \mathcal{O},\forall w\in\otimes^kV^*\,.$$
We can construct such maps as follows. Let $k=2\ell$ and let $\pi\in\operatorname{Perm}(2\ell)$ be a permutation of
the integers from
$1$ to
$2\ell$. Define
\begin{equation}\label{eqn-2.b}
\psi_\pi(v^1,\dots,v^{2\ell}):=\langle v^{\pi(1)},v^{\pi(2)}\rangle\cdot\dots\cdot
\langle v^{\pi(2\ell-1)},v^{\pi(2\ell)}\rangle\,.
\end{equation}We show
$\psi_\pi$ is an orthogonal invariant by computing\medbreak\qquad$\psi_\pi(\Theta v^1,\dots,\Theta v^{2\ell})$
$=\langle \Theta v^{\pi(1)},\Theta v^{\pi(2)}\rangle\cdot\dots\cdot\langle \Theta v^{\pi(2\ell-1)},\Theta v^{\pi(2\ell)}\rangle$
\smallbreak\qquad\quad
$=\langle v^{\pi(1)},v^{\pi(2)}\rangle\cdot\dots\cdot\langle v^{\pi(2\ell-1)},v^{\pi(2\ell)}\rangle$
$=\psi_\pi(v^1,\dots,v^{2\ell})$.\medbreak\noindent
Since $\psi_\pi$ is a multi-linear map, it extends naturally to a linear orthogonal invariant mapping
$\otimes^{2\ell}V$ to $\mathbb{R}$. We refer to \cite{W46} (see Theorem~2.9.A on page 53) for the proof of the following result:
\begin{theorem}\label{thm-2.1}
The
space of linear orthogonal invariants of $\otimes^{2k}V^*$ is spanned by the maps $\psi_\pi$ of {\rm Equation~(\ref{eqn-2.b})}.
\end{theorem}

In geodetic polar coordinates, we set $g_{ij}(P)=\delta_{ij}$ and
$g_{ij/k}(P)=0$; the remaining derivatives of the metric can
be expressed in terms of the covariant derivatives of the curvature
tensor at $P$. The following result \cite{ABP73} is then a direct
consequence of Theorem \ref{thm-2.1}; the extension from scalar to
symmetric 2-form valued invariants is immediate. Lemma \ref{lem-1.1}
and Lemma \ref{lem-1.3} follow directly the following Lemma after
using the curvature identities to eliminate redundancies and we
refer the reader to those results to illustrate exactly what is
meant by Lemma \ref{lem-2.2}:

\begin{lemma}\label{lem-2.2}
 All scalar invariants and all symmetric 2-form valued invariants which are given by a local formula in the derivatives of the metric
and which are homogeneous of order
$n$ arise by contracting indices in pairs in monomial expressions of weight $n$ in the covariant derivatives of the curvature
tensor.
\end{lemma}

\subsection{The restriction map}\label{sect-2.4}
Let $(N,g_N)$ be a Riemannian manifold of dimension $m-1$. Let
$M=N\times S^1$ and let $g_M=g_N+d\theta^2$ where $\theta$ is the
usual periodic parameter on the circle. Let $\theta_0$ be the
basepoint of the circle; since $(S^1,d\theta^2)$ is a homogeneous
space, the choice of the basepoint plays no
role. If $y\in N$, we let $i(y):=(y,\theta_0)\in M$. If
$Q\in\mathcal{I}_{m,n}$ or if $Q\in\mathcal{I}_{m,n}^2$, then we set
\begin{equation}\label{eqn-2.c}
r(Q)(g_N,y):=i^*Q(g_M,i(y))\,;
\end{equation}
(we have to restrict this tensor to $N\times\{\theta_0\}$).
This defines natural maps
$$r:\mathcal{I}_{m,n}\rightarrow\mathcal{I}_{m-1,n}\quad\text{and}\quad
r:\mathcal{I}_{m,n}^2\rightarrow\mathcal{I}_{m-1,n}^2\,.$$

 Assertion (1) of Theorem \ref{thm-1.1} and of Theorem \ref{thm-1.2} will
follow from:
\begin{lemma}\label{lem-2.3}
We have $r:\mathcal{I}_{m,n}\rightarrow\mathcal{I}_{m-1,n}\rightarrow0$ and
$r:\mathcal{I}_{m,n}^2\rightarrow\mathcal{I}_{m-1,n}^2\rightarrow0$.
\end{lemma}

\begin{proof} By Lemma \ref{lem-2.2}, all local invariants are given in terms of contractions of indices of various monomials of
weight $n$ in the covariant derivatives of the curvature tensor. Instead of letting the indices range from
$1$ to
$m$ in the contractions of indices which define
$Q$, we let the indices range from
$1$ to
$m-1$ in defining $r(Q)$ since the metric is flat in the last direction. Thus, for example, as noted above we have:
$$
\tau_m:=\sum_{i,j=1}^mR_{ijji}\quad\text{then}\quad
r(\tau_m)=\tau_{m-1}=\sum_{i,j=1}^{m-1}R_{ijji}\,.
$$
This is, of course, implicit in the notation that we used in
Equation (\ref{eqn-2.a}) in defining the scalar curvature in the
first instance. The dimension $m$ appears implicitly in the range of
summation and the formula is ``universal'' over all dimensions in
that respect, i.e. $r(\tau_m)=\tau_{m-1}$. Thus we usually don't
subscript but simply talk of the scalar curvature $\tau$ without
mentioning the underlying dimension $m$. We may choose a spanning
set for $\mathcal{I}_{m-1,n}$ or $\mathcal{I}_{m-1,n}^2$
similar to those given in Lemma \ref{lem-1.1} and in Lemma
\ref{lem-1.3} which involves contracting indices in covariant
derivatives of the curvature tensor. The desired lift to
$\mathcal{I}_{m,n}$ or to $\mathcal{I}_{m,n}^2$ is then obtained by
letting the indices range from $1$ to $m$ instead of from $1$ to
$m-1$. This lift is, of course, not unique and is exactly measured
by $\ker(r)$ which gives the universal relations satisfied in
dimension $m-1$ which are not satisfied in dimension $m$.
\end{proof}

We used the tensor calculus to show that $r$ is surjective. We now return to the non-invariant formulation to continue our
study. We may always restrict to coordinate systems $\vec x$ which are normalized at the point $P$ so that
\begin{equation}\label{eqn-2.d}
g_{ij}(\vec x,g,P)=\delta_{ij}\quad\text{and}\quad g_{ij/k}(\vec x,g,P)=0\,.
\end{equation}
We let $\tilde Q_m:=\mathbb{R}[g_{ij/\alpha}]_{|\alpha|\ge2}$ be the polynomial algebra in the jets of the metric of order at
least 2. One can use a partition of unity and Taylor series to derive the following result:
\begin{lemma}\label{lem-2.4}
If $0\ne Q\in\tilde Q_m$, then there exists $(\vec x,g,P)$ so that $\vec x$ satisfies the normalizations of Equation
(\ref{eqn-2.d}) and so that
$Q(\vec x,g,P)\ne0$.
\end{lemma}

We note that Lemma \ref{lem-2.4} is {\bf not} true if we work with the Riemann curvature tensor. There are ``hidden'' and
non-obvious relations that do not follow from the usual $\mathbb{Z}_2$ symmetries and the generalized Bianchi identities that
are dimension specific - that is the whole point, of course, of the relations given in Lemma \ref{lem-1.2} and in Lemma
\ref{lem-1.4}. And it is Lemma \ref{lem-2.4} that will be crucial in our discussion.

Let $A=g_{i_1j_1/\alpha_1}\cdot\cdot\cdot g_{i_\ell
j_\ell/\alpha_\ell}$ be a monomial of
$\tilde Q_m$. We define
$$\operatorname{deg}_k(A):=\delta_{i_1,k}+\delta_{j_1,k}+\alpha_1(k)+...+\delta_{i_\ell,k}
+\delta_{j_\ell,k}+\alpha_\ell(k)$$ to be the number of times that
the index $k$ appears in $A$. We extend this notion to the context of symmetric
$2$-form valued invariants by defining:
$$
\operatorname{deg}_k(Adx^{i_{\ell+1}}\circ dx^{j_{\ell+1}}):=
\operatorname{deg}_k(A)+\delta_{i_{\ell+1},k}+\delta_{j_{\ell+1},k}\,.
$$
 Set $r_1(A)=A$ if
$\operatorname{deg}_m(A)=0$ and $r_1(A)=0$ if
$\operatorname{deg}_m(A)>0$ to define a polynomial map $r_1:\tilde
Q_{m}\rightarrow\tilde Q_{m-1}$. Assertion (2) of Theorem
\ref{thm-1.1} and Assertion (2) of Theorem \ref{thm-1.2} will follow
Lemma \ref{lem-2.3} and from:
\goodbreak\begin{lemma}\label{lem-2.5}
\ \begin{enumerate}
\item  If $Q\in\mathcal{I}_{m,n}$ or if $Q\in\mathcal{I}_{m,n}^2$, then $r_1(Q)=r(Q)$.
\smallbreak\item If $Q\in\mathcal{I}_{m,n}\cap\ker(r)$ or if
$Q\in\mathcal{I}_{m,n}^2\cap\ker(r)$, then $\operatorname{deg}_k(A)\ge2$ for
$1\le k\le m$ for every monomial $A$ of $Q$. \smallbreak\item If
$m>n$, then
$\ker(r:\mathcal{I}_{m,n}\rightarrow\mathcal{I}_{m-1,n})=\{0\}$.
\smallbreak\item If $m=n$, if $Q\in\ker(r)\cap\mathcal{I}_{m,n}$,
and if $A$ is a monomial of $Q$, then
\smallbreak\quad$\operatorname{deg}_k(A)=2\quad\text{and}\quad|\alpha_a|=2\quad\text{
for }\quad 1\le k\le m\text{ and }1\le a\le\ell$.
\smallbreak\item If $m>n+1$, then $\ker(r:\mathcal{I}_{m,n}^2\rightarrow\mathcal{I}_{m-1,n}^2)=\{0\}$.
\smallbreak\item If $m=n+1$ if $Q\in\ker(r)\cap\mathcal{I}_{m,n}^2$, and if $A$ is a monomial of $Q$, then
\smallbreak\quad$\operatorname{deg}_k(A)=2\quad\text{and}\quad|\alpha_a|=2\quad\text{
for }\quad 1\le k\le m\text{ and }1\le a\le\ell$.
\end{enumerate}
\end{lemma}

\begin{proof} Assertion (1) gives an algebraic reformulation of the geometric definition given in Equation (\ref{eqn-2.c}) and
is immediate from that definition; the metric on $N\times S^1$ is flat in the final direction; we also set $e^i\circ e^j=0$ if
either $i$ or $j$ is the final index as we have to restrict the tensor to the submanifold.

Let $r(Q)=0$. By Lemma
\ref{lem-2.4}, we may identify the local formula defined by $Q$ with the polynomial $Q\in\mathcal{Q}_m$. It then follows that
$\operatorname{deg}_m(A)>0$ for every monomial $A$ of $Q$. Let
$y=(x^1,...,x^{m-1},-x^m)$, we see $\operatorname{deg}_m(A)$ is even and hence $\operatorname{deg}_m(A)\ge2$. Since
$Q$ is invariant under coordinate permutations, Assertion (2) follows.

Let $0\ne Q\in\mathcal{I}_{m,n}\cap\ker(r)$.  Let
$A=g_{i_1j_1/\alpha_1}\cdot\cdot\cdot g_{i_\ell j_\ell/\alpha_\ell}$
be a monomial of $Q$. Since $|\alpha_a|\ge2$, we have
\begin{equation}\label{eqn-2.e}
2\ell\le\sum_{a=1}^\ell|\alpha_\ell|=n\,.\end{equation}
By Assertion (2) we have $\operatorname{deg}_k(A)\ge2$ for every $k$. Thus
\begin{equation}\label{eqn-2.f}
\begin{array}{l}
\displaystyle2m\le\sum_{1\le k\le m}\operatorname{deg}_k(A)=\sum_{a=1}^\ell\sum_{k=1}^m
\left\{\delta_{i_a,k}+\delta_{j_a,k}+\alpha_a(k)\right\}\\
\displaystyle\qquad=\sum_{a=1}^\ell\{1+1+|\alpha_a|\}=2\ell+n\le n+n=2n\,.\vphantom{\vrule height 16pt}\end{array}\end{equation}This shows that $m\le n$ and proves Assertion (3). Furthermore, if $m=n$, all the inequalities in Equation (\ref{eqn-2.e})
and in Equation (\ref{eqn-2.f}) must have been equalities; this establishes Assertion (4).

Similarly let $0\ne Q\in\mathcal{I}_{m,n}^2\cap\ker(r)$ and let $A$ be a monomial of $Q$. Express
$$A=g_{i_1j_1/\alpha_1}\cdot\cdot\cdot g_{i_\ell j_\ell/\alpha_\ell}dx^{i_{\ell+1}}\circ dx^{j_{\ell+1}}\,.$$
We estimate similarly:
\begin{eqnarray}
&&2\ell\le\sum_{a=1}^{\ell}|\alpha_\ell|=n,\label{eqn-2.g}\\
&&\displaystyle2m\le\sum_{k=1}^{m}\operatorname{deg}_k(A)=\sum_{a=1}^\ell\sum_{k=1}^m\left\{\delta_{i_a,k}+\delta_{j_a,k}+\alpha_a(k)\right\}+2\label{eqn-2.h}\\
&&\displaystyle\qquad=\sum_{a=1}^{\ell}\{1+1+|\alpha_a|\}+2=2\ell+n+2\le 2n+2\,.\nonumber
\end{eqnarray}
Again, this is not possible if $m>n+1$ which establishes Assertion (5). If $m=n+1$, all the equalities must have been
equalities and the desired result follows.
\end{proof}

\subsection{H. Weyl's second theorem}\label{sect-2.5}
Let $(V,\langle\cdot,\cdot\rangle)$ be an inner product space of dimension $m$.
A typical relation among scalar products is the following which involves
$m+1$ vectors $\{v^0,...,v^m\}$ and $m+1$ vectors $\{w^0,...,w^m\}$. One necessarily has:
\begin{equation}\label{eqn-2.i}
\det\left(\begin{array}{llll}
\langle v^0,w^0\rangle&\langle v^0,w^1\rangle&\dots&\langle v^0,w^m\rangle\\
\langle v^1,w^0\rangle&\langle v^1,w^1\rangle&\dots&\langle v^1,w^m\rangle\\
\dots&\dots&\dots&\dots\\
\langle v^m,w^0\rangle&\langle v^m,w^1\rangle&\dots&\langle v^m,w^m\rangle
\end{array}\right)=0\,.
\end{equation}

One also has \cite{W46} (see Theorem~2.17.A page 75)

\begin{theorem}\label{thm-2.2}
Every relation among scalar products is an algebraic consequence of the relations given above in
{\rm Equation~(\ref{eqn-2.i})}.
\end{theorem}

\subsection{Proof of Theorem \ref{thm-1.1}}\label{sect-2.6} Let $m=2\bar m$ be even. We introduce formal variables $g_{ij/kl}\in
S^2\otimes S^2$ for $1\le i,j,k,l\le m$. If
$Q\in\ker(r:\mathcal{I}_{m,m}\rightarrow\mathcal{I}_{m-1,m})$, then we have shown that in Lemma \ref{lem-2.5} that $Q$ can be
regarded as a polynomial of degree $\bar m$ in
$\mathbb{R}[g_{ij/kl}]$. Let $S^2$ denote the space of symmetric $2$ tensors. Since $g_{ij/kl}\in S^2\otimes S^2$, we can regard
$Q$ as a linear orthogonal invariant on $\otimes^{\bar m}\{S^2\otimes S^2\}$. Such an orthogonal invariant extends naturally to
the full tensor algebra to be zero on the orthogonal complement of $\otimes^{\bar m}\{S^2\otimes S^2\}$ and hence H. Weyl's
theorem applies where the dimension of the underlying vector space is $m-1$ not $m$. Since the restriction of $Q$ to the lower
dimensional setting vanishes, we can apply Theorem
\ref{thm-2.2} to express
$Q$ as a linear combination of invariants of the form
\begin{eqnarray*}
A_\sigma&=&g_{i_1i_2/i_3i_4}\cdot\cdot\cdot g_{i_{2m-3}i_{2m-2}/i_{2m-1}i_{2m}}\\
&\times&
g(dx^{i_{\sigma_1}}\wedge dx^{i_{\sigma_2}}\wedge\cdot\cdot\cdot\wedge dx^{i_{\sigma_m}},
dx^{i_{\sigma_{m+1}}}\wedge\cdot\cdot\cdot
\wedge dx^{i_{\sigma_{2m}}})
\end{eqnarray*}
where $\sigma$ is a permutation of $\{1,...,2m\}$. If $i_1=i_{\sigma_a}$ for
some index $a$ with $1\le a\le m$, then necessarily
$i_2=i_{\sigma_b}$ for some index $b$ with $m+1\le b\le 2m$ since $g_{i_1i_2/i_3i_4}$ is symmetric in the indices
$\{i_1,i_2\}$ where as the wedge product is anti-symmetric. By permuting the
indices $\{i_1,i_2\}$ if necessary, we may therefore assume
$i_1=\sigma_{a_1}$ and
$i_2=\sigma_{b_1}$ for $1\le a_1\le m$ and $m+1\le b_1\le 2m$. This implies we can write
\begin{eqnarray*}
A_\sigma&=&g_{i_1j_1/i_2j_2}...g_{i_{m-1}j_{m-1}/i_mj_m}\\
&\times&g(dx^{i_{\rho_1}}\wedge...\wedge dx^{i_{\rho_m}},
dx^{j_{\varrho_1}}\wedge...\wedge dx^{j_{\varrho_m}})
\end{eqnarray*}
where $\rho$ and $\varrho$ are permutations of $m$ indices. Reordering the factors then yields
\begin{eqnarray*}
A_\sigma&=&\pm g_{i_1j_1/i_2j_2}...g_{i_{m-1}j_{m-1}/i_mj_m}\\
&\times&g(dx^{i_1}\wedge...\wedge dx^{i_m},
dx^{j_1}\wedge...\wedge dx^{j_m})\,.
\end{eqnarray*}
This shows $\dim\{\ker(r:\mathcal{I}_{m,m}\rightarrow\mathcal{I}_{m-1,m})\}\le1$. Since $r(E_{m,m})=0$ and $E_{m,m}$ is
non-trivial, Assertion (3) of Theorem \ref{thm-1.1} follows.\hfill\qed

\subsection{Proof of Theorem \ref{thm-1.2}}\label{sect-2.7}
The proof of Theorem \ref{thm-1.2} (3) is essentially the same. The crucial feature is, of course, that we have
eliminated the higher order jets of the metric and only have to deal with second derivatives. The dimension of the
underlying vector space is now
$m=2\bar m$ rather than
$m-1$. Let
$Q\in\mathcal{I}_{m+1,m}^2$. We can express $Q=Q_{uv}dx^u\circ dx^v$ where $Q_{uv}\in\mathbb{R}[g_{ij/kl}]$ is homogeneous of
degree $\bar m$. Since $r(Q)=0$, we may express $Q$ as a linear combination of invariants of the form:
\begin{eqnarray*}
A_\sigma&=&g_{i_1i_2/i_3i_4}\cdot\cdot\cdot g_{i_{2m-3}i_{2m-2}/i_{2m-1}i_{2m}}dx^{i_{2m+1}}\circ
dx^{i_{2m+2}}\\
&\times&
g(dx^{i_{\sigma_1}}\wedge dx^{i_{\sigma_2}}\wedge\cdot\cdot\cdot\wedge dx^{i_{\sigma_{m+1}}},
dx^{i_{\sigma_{m+2}}}\wedge\cdot\cdot\cdot
\wedge dx^{i_{\sigma_{2m+2}}})\,.
\end{eqnarray*}
The same symmetry argument used to establish Theorem \ref{thm-1.1} then shows in fact we are dealing with
\begin{eqnarray*}
A_\sigma&=&\pm g_{i_1j_1/j_2i_2}\cdot\cdot\cdot g_{i_{m-1}j_{m-1}/j_mi_m}dx^{i_{m+1}}\circ
dx^{j_{m+1}}\\
&\times&
g(dx^{i_1}\wedge dx^{i_2}\wedge\cdot\cdot\cdot\wedge dx^{i_{m+1}},
dx^{j_1}\wedge dx^{j_2}\wedge\cdot\cdot\cdot\wedge dx^{j_{m+1}})\,.
\end{eqnarray*}
Again, this shows $\dim\{\ker(r:\mathcal{I}_{m+1,m}^2\rightarrow\mathcal{I}_{m,m}^2)\}\le1$. The desired result then follows as
$ T_{m+1,m}^2\in\ker(r:\mathcal{I}_{m+1,m}^2\rightarrow\mathcal{I}_{m,m}^2)$ is
non-trivial.\hfill\qed

\section*{Acknowledgments}
Research of P. Gilkey partially supported by project MTM2009-07756
(Spain), by INCITE09 207 151 PR (Spain), and by DFG PI 158/4-6
(Germany).  Research of J. H. Park and K. Sekigawa was supported by
Basic Science Research Program through the National Research
Foundation of Korea (NRF) funded by the Ministry of Education,
Science and Technology.

\end{document}